\documentclass[11pt]{amsart}
\usepackage[T1]{fontenc}
\usepackage[utf8]{inputenc}
\usepackage{lmodern}
\usepackage{amsmath,amssymb,amsthm,mathtools}
\usepackage{enumitem}

\numberwithin{equation}{section}

\newtheorem{theorem}{Theorem}[section]
\newtheorem{lemma}[theorem]{Lemma}
\newtheorem{proposition}[theorem]{Proposition}
\newtheorem{corollary}[theorem]{Corollary}
\theoremstyle{definition}
\newtheorem{example}[theorem]{Example}

\theoremstyle{remark}
\newtheorem{remark}[theorem]{Remark}

\makeatletter
\newcommand{\loccite}[2]{[\@ifundefined{b@#2}{\textbf{??}}{\csname b@#2\endcsname}, #1]}
\makeatother

\DeclareMathOperator{\conv}{conv}
\DeclareMathOperator{\Int}{Int}
\newcommand{\RR}{\mathbb{R}}
\newcommand{\CC}{\mathbb{C}}

\newcommand{\ZZ}{\mathbb{Z}}
\newcommand{\TT}{\mathbb{T}}
\newcommand{\DD}{\mathbb{D}}
\newcommand{\BH}{\mathcal B(H)}

\begin{document}
\title
[Invariant subspaces and numerical range]{Invariant subspaces for operators with spectrum containing the boundary of the numerical range}

\author{Vladimir M\"uller}
\address{Institute of Mathematics\\
Czech Academy of Sciences\\
\v Zitna Str. 25, 115 67 Prague\\
Czech Republic}
 
\email{muller@math.cas.cz}

\author{Yuri Tomilov}
\address{
Institute of Mathematics\\
Polish Academy of Sciences\\
\'Sniadeckich 8\\
00-656 Warsaw, Poland\\
Faculty of Mathematics and Computer Science\\
Nicolaus Copernicus University\\
Chopin Street 12/18\\
87-100 Toru\'n, Poland
}
\email{ytomilov@impan.pl}

\thanks{The research was supported by GA CR/NCN grant 25-15444K. The first author was supported by the Czech Academy of Sciences (RVO:67985840). The second author was supported by the NCN grant Weave-Unisono, 2024/06/Y/ST1/00044. The second author was also partially supported by the NCN grant UMO-2023/49/B/ST1/01961 and the NAWA/NSF grant BPN/NSF/2023/1/00001.}

\begin{abstract}
We prove that a bounded Hilbert space operator has a nontrivial invariant
subspace whenever its spectrum contains the topological boundary of its
numerical range. We also give variants of this result and establish related
statements. The criterion is applied to hyponormal, cohyponormal and Toeplitz
operators. In these classes, convexity of the polynomial hull of the spectrum
already suffices. Several examples are given that are not covered by
well-known existence criteria.
\end{abstract}

\maketitle
\section{Main result}
Throughout the paper, $H$ denotes a complex infinite-dimensional Hilbert space and $\BH$ denotes the algebra of all bounded linear operators on $H$. The invariant subspace problem is the famous question whether each operator in $\BH$ has a nontrivial invariant subspace. In spite of the great effort of many mathematicians, the problem is still open. There are many positive results for various classes of operators including compact, selfadjoint, subnormal operators, contractions with spectrum containing the unit circle, etc.
However, there are very few criteria applicable to general operators. For example, while classical Lomonosov's theorem has an elegant formulation,
it is not easy to apply and produce explicit examples of operators with nontrivial invariant subspaces.

A useful source of such criteria is provided by spectral-set estimates.
We offer a simple observation whose point is to bring the numerical range into the problem through Crouzeix's theorem, stating that the closure of the numerical range of \(T\in\mathcal B(H)\) is a spectral set for \(T\). Recall that
for $T \in \BH$ its numerical range $W(T)$ is defined by
$$
W(T):=\{\langle Tx,x\rangle: x\in H, \|x\|=1\},
$$
while a compact set $X\subset\CC$ is a $K$-spectral set for $T$ if $\sigma(T)\subset X$ and
\begin{equation}\label{spset}
    \|r(T)\|\le K \|r\|_X,
\end{equation}
for every rational function $r$ with poles off $X$, where
\[
    \|r\|_X:=\sup\{|r(z)|:z\in X\}.
\]
When the value of $K$ is irrelevant, we simply say that $X$ is a spectral set for $T$.
For convex, and hence polynomially convex, sets $X$, rational functions can be replaced in \eqref{spset} by polynomials.
The spectral-set estimate below was proved initially in \cite{Crou} with \(K=11.08\), and improved to \(K=1+\sqrt2\) in \loccite{Theorem~3.1}{CP}.
\begin{theorem}\label{crouzeix} 
If $T \in \BH$, then $\overline{W(T)}$ is a $K$-spectral set for $T$, with $K=1+\sqrt2$.
\end{theorem}
Theorem \ref{crouzeix} will be combined 
with the following result in \loccite{Theorem~A}{AM} guaranteeing existence
of nontrivial invariant subspaces for operators having closed unit disc as a spectral set, under
peripheral spectrum assumptions.
\begin{theorem}\label{thm:AM}
Let $T$ be a polynomially bounded operator on a Hilbert space. If the unit circle $\mathbb T$
belongs to $\sigma(T)$, then $T$ has a nontrivial invariant subspace.
\end{theorem}

Now Theorems \ref{crouzeix} and \ref{thm:AM} put together lead to the following invariant-subspace criterion. It is formulated only in terms of two basic sets associated with an operator $T\in\BH$: the spectrum $\sigma(T)$ and the numerical range $W(T)$. It is thus directly checkable in many concrete situations. 

To elaborate this result,
 denote by $\partial A$ the topological boundary of a set $A\subset\CC,$ and by $\Int A$ its interior. Let, in addition, $\DD=\{z\in\CC: |z|<1\}.$ 
The algebra $A(K)$ of functions continuous on a compact $K\subset \mathbb C$ and analytic in $\operatorname{Int}K$ will play an essential role.

\begin{theorem}\label{t1}
Let $T\in\BH$. Suppose that $\sigma(T)\supset\partial W(T)$. Then $T$ has a nontrivial invariant subspace.
\end{theorem}

\begin{proof}
It is well-known that $W(T)$ is a convex subset of the complex plane. Suppose first that the interior $\Int W(T)$ is empty.
So $W(T)$ is contained in a straight line. Thus there exist complex numbers $\eta, a$ such that $|\eta|=1$ and 
$\{\eta z+a:z\in W(T)\}$ is contained in the real axis.

Let $S:=\eta T+ aI$, where $I$ is the identity operator on $H$. Then $W(S)\subset\RR$, and so $S$ is a selfadjoint operator.
Hence $S$ has a nontrivial invariant subspace, which is also invariant for the operator $T=\eta^{-1}(S- aI)$.

Suppose now that $\Int W(T)\ne\emptyset$, and put $G=\Int W(T)$. Since $W(T)$ is convex, one has
$\overline{G}=\overline{W(T)}$ and $\partial G=\partial W(T)$. By the spectral inclusion
$\sigma(T)\subset\overline{W(T)}$, it follows that $\sigma(T)\subset\overline{G}$.

By 
Theorem \ref{crouzeix},
$\overline G$ is a $K$-spectral set for $T$ (with $K=1+\sqrt2$). Using Mergelyan's theorem and taking into account the convexity of $G,$ we infer 
that 
 the polynomial calculus for $T$ extends by continuity to the whole of $A(\overline G)$
and we have
$$
\|f(T)\|\le K\|f\|_{\overline G}
$$
for all $f\in A(\overline G).$

The final step essentially follows \cite[Remark]{KL}. We give a short
argument to make explicit how the geometry supplied by the numerical-range
assumption and the conformal reduction allow one to apply Theorem
\ref{thm:AM}.
By Carath\'eodory's theorem there exists a conformal mapping $\phi:G\to\DD$
which extends to a homeomorphism
$$
\phi:\overline G\to\overline{\DD}.
$$
Thus $\phi\in A(\overline G)$, and we may define a bounded operator
$$
S:=\phi(T).
$$
We claim that $\mathbb T \subset\sigma(S)$. Indeed, let $\lambda\in\partial G$. Then $\lambda\in\partial W(T)\subset\sigma(T)$ by assumption. Choose polynomials $(q_n)_{n\ge1}$ such that $q_n\to\phi$ uniformly on $\overline G$ as $n\to\infty$. By the polynomial spectral mapping theorem,
$$
q_n(\lambda)\in\sigma(q_n(T)),\qquad n\ge1.
$$
Since $q_n(T)\to S$ in norm and $q_n(\lambda)\to\phi(\lambda)$ as $n \to \infty$, upper semi-continuity of the spectrum
yields
$$
\phi(\lambda)\in\sigma(S).
$$
As $\phi(\partial G)=\mathbb T$, this proves $\mathbb T \subset\sigma(S)$.

For each polynomial $p$ we have
$$
\|p(S)\|=\|(p\circ \phi)(T)\|\le K\|p\circ\phi\|_{\overline G}=K\|p\|_{\overline{\DD}}.
$$
Hence $S$ is polynomially bounded. Since $\mathbb T \subset\sigma(S)$, Theorem \ref{thm:AM} implies that $S$ has a nontrivial invariant subspace $M$.
The inverse map $\phi^{-1}:\overline{\DD}\to\overline G$ belongs to $A(\overline{\DD})$. By polynomial boundedness of $S$, the polynomial calculus of $S$ extends to $A(\overline{\DD})$, and so $\phi^{-1}(S)$ is well-defined. If $(p_n)_{n\ge1}$ are polynomials converging uniformly to $\phi^{-1}$ on $\overline{\DD}$, then
$$
p_n(S)=(p_n\circ\phi)(T)\longrightarrow T \qquad \text{in norm}
$$
as $n\to\infty$, because $p_n\circ\phi\to\phi^{-1}\circ\phi$ uniformly on $\overline G$. Hence $T=\phi^{-1}(S)$. Since $M$ is invariant under $S$, it is invariant under polynomials in $S$ and therefore under $\phi^{-1}(S)=T$. Thus $T(M)\subset M$.
\end{proof}

For a compact set $K\subset\CC$, we write
\[
    \widehat K=\{z\in\CC: |p(z)|\le \sup_K |p|\hbox{ for every polynomial }p\}
\]
for its polynomial hull. We shall use the following elementary reformulation of the hypothesis in
Theorem \ref{t1}. In the non-degenerate case it identifies the condition with
a statement about the polynomial hull of the spectrum.

\begin{lemma}\label{lem:hull-numerical-range}
Let $T\in\BH$ and suppose that $\operatorname{Int} W(T)\ne\emptyset$.
Then
\[
        \partial W(T)\subset\sigma(T)
        \quad\Longleftrightarrow\quad
        \widehat{\sigma(T)}=\overline{W(T)}.
\]
\end{lemma}

\begin{proof}
Put $C=\overline{W(T)}$. Since $W(T)$ is convex and has nonempty
interior, $C$ is a compact convex set with $\partial W(T)=\partial C$.
Moreover $\sigma(T)\subset C$. Suppose first that
$\partial W(T)\subset\sigma(T)$, that is, $\partial C\subset\sigma(T)$.
Then
\[
        \widehat{\partial C}\subset\widehat{\sigma(T)}\subset\widehat C=C .
\]
Since $C$ is compact and convex with nonempty interior,
$\widehat{\partial C}=C$: indeed, $C$ is obtained from $\partial C$ by
filling in the bounded component $\operatorname{Int} C$ of
$\CC\setminus\partial C$. Hence $\widehat{\sigma(T)}=C$.

Conversely, suppose that $\widehat{\sigma(T)}=C$. For every compact plane
set $K$, the set $\widehat K$ is obtained from $K$ by filling in the bounded
components of $\CC\setminus K$.  In particular, one has
$\partial\widehat K\subset K$. Applying this to $K=\sigma(T)$ gives
\[
        \partial W(T)=\partial C=\partial\widehat{\sigma(T)}
        \subset\sigma(T).
\]
\end{proof}

If $T$ is a contraction satisfying $\TT\subset\sigma(T)$, then $W(T)\subset\overline{\DD}$ and, since $\sigma(T)\subset\overline{W(T)}$, the closed convex set $\overline{W(T)}$ contains $\TT$. Hence
$$
\overline{W(T)}=\overline{\DD},
$$
and therefore $\partial W(T)=\TT\subset\sigma(T)$. Thus Theorem \ref{t1} is a formal generalization of the famous result in \loccite{Corollary~1.2}{BChP} stating that any contraction with spectrum containing the unit circle has a nontrivial invariant subspace.

\subsection{A two-sided variant}
\label{subsec:two-sided-variant}

The proof of Theorem \ref{t1} also gives a two-sided conclusion in a simple
additional situation. This will be useful to keep in mind when comparing the
one-sided invariant-subspace conclusion with the translation-invariant, or
bi-invariant, problem.

\begin{corollary}\label{cor:two-sided-zero-away}
Let $T\in\BH$ be invertible. Suppose that
\[
    \partial W(T)\subset \sigma(T)
    \qquad\hbox{and}\qquad
    0\notin \overline{W(T)} .
\]
Then there exists a nontrivial subspace of $H$ which is invariant under
both $T$ and $T^{-1}$.
\end{corollary}

\begin{proof}
If $\Int W(T)=\emptyset$, then, as in the proof of Theorem \ref{t1}, an affine
transform of $T$ is selfadjoint. Since $T$ is invertible, the spectral theorem
then gives a nontrivial reducing subspace for $T$, unless $T$ is scalar. In the
scalar case any nontrivial subspace will do. In either case the subspace
is invariant under both $T$ and $T^{-1}$.

Assume now that $\Int W(T)\ne\emptyset$ and put $G=\Int W(T)$. In the proof of
Theorem \ref{t1} we constructed a conformal map
$\phi:G\to\mathbb D$, extending homeomorphically to the closures, and the
operator $S=\phi(T)$. The subspace obtained there is invariant under the
$A(\overline{\mathbb D})$-functional calculus of $S$: indeed this follows by applying
polynomial approximation to functions in $A(\overline{\mathbb D})$. Since $0\notin\overline{G}$, the
function $z\mapsto z^{-1}$ belongs to $A(\overline G)$. Hence
\[
    T^{-1}=(z^{-1})(T)=g(S),
    \qquad
    g=(z^{-1})\circ\phi^{-1}\in A(\overline{\mathbb D}).
\]
Thus the same subspace is invariant under $T^{-1}$ as well as under $T$.
\end{proof}

\section{Hyponormal operators}
In this section, we elaborate applications of Theorem \ref{t1} to the study of hyponormal operators.
Recall that
an operator $T\in\BH$ is called hyponormal if $T^*T\ge TT^*$. Equivalently, $\|Tx\|\ge\|T^*x\|$ for all $x\in H$.
The class of hyponormal operators is classical and has a rich invariant-subspace and spectral theory; see, for instance, \cite[Ch.~IV]{MP}.
However, it is still not known whether each hyponormal operator has a nontrivial invariant subspace.

We recall several properties of hyponormal operators needed below.
Letting $\conv A$ stand for the convex hull of $A \subset \mathbb C,$ note first that each hyponormal operator satisfies
\begin{equation}\label{convexoid-hyp}
\overline{W(T)}=\conv\sigma(T),
\end{equation}
see \loccite{Theorem~2}{S}. In other words, every hyponormal operator is convexoid.
In view of \eqref{convexoid-hyp}, the hypothesis of Theorem \ref{t1}
is equivalent to
\begin{equation}\label{boundary-conv-spectrum}
\partial\conv\sigma(T)\subset\sigma(T).
\end{equation}
Indeed, if $W(T)$ has nonempty interior, then
\[
\partial W(T)=\partial\overline{W(T)}=\partial\conv\sigma(T),
\]
while, if $W(T)$ has empty interior, then
\[
\partial W(T)=\overline{W(T)}=\conv\sigma(T),
\]
and the two conditions are again the same. By Lemma
\ref{lem:hull-numerical-range}, in the non-degenerate case the same condition
can also be written as
\begin{equation}\label{hull-conv-spectrum}
        \widehat{\sigma(T)}=\conv\sigma(T).
\end{equation}

Before stating the corollary, let us isolate the simple boundary set used below. For an operator $T$, denote by $S(T)$ the set of all points $\lambda\in\partial W(T)$ which belong to the closure of a nontrivial straight line segment contained in $\partial W(T)$. Thus $S(T)$ is the flat part of the boundary of the numerical range, together with its endpoints. In the convexoid situation \eqref{convexoid-hyp}, the inclusion
\[
    S(T)\subset\sigma(T)
\]
is a simple sufficient spectral condition. The case $S(T)=\emptyset$ is the particularly transparent special case where there are no nontrivial straight pieces on $\partial W(T)$.

\begin{corollary}\label{c1}
Let $T\in\BH$ be a hyponormal operator. If
$$
S(T)\subset\sigma(T),
$$
then $T$ has a nontrivial invariant subspace. In particular, this holds whenever $S(T)=\emptyset$.
Equivalently, the same conclusion holds under the condition
$$
\partial\conv\sigma(T)\subset\sigma(T).
$$
\end{corollary}

\begin{proof}
It is enough, in view of \eqref{boundary-conv-spectrum} and Theorem \ref{t1}, to prove that $S(T)\subset\sigma(T)$ implies
$$
\partial\conv\sigma(T)\subset\sigma(T).
$$
Put $K=\sigma(T)$. By \eqref{convexoid-hyp}, $\partial\conv K=\partial\overline{W(T)}$, and $W(T)$ and $\overline{W(T)}$ have the same boundary. Let $\lambda\in\partial\conv K$. If $\lambda$ is an extreme point of the compact convex set $\conv K$, then $\lambda\in K$ (e.g. by the elementary finite-dimensional form of the Krein--Milman theorem). If $\lambda$ is not an extreme point, choose a supporting line to $\conv K$ at $\lambda$. The corresponding face contains a nontrivial line segment in $\partial\conv K=\partial W(T)$ whose closure contains $\lambda$. Hence $\lambda\in S(T)\subset\sigma(T)=K$. Thus $\partial\conv\sigma(T)\subset\sigma(T)$, and the conclusion follows. The converse implication from $\partial\conv\sigma(T)\subset\sigma(T)$ to $S(T)\subset\sigma(T)$ is immediate, since $S(T)\subset\partial W(T)=\partial\conv\sigma(T)$.
\end{proof}

\begin{remark}\label{r-convex-spectrum}
A convenient sufficient condition is formulated in terms of the polynomial
hull. If $K\subset\CC$ is compact, then $\widehat K\subset\conv K$.
Hence, if $\widehat K$ is convex, then $\widehat K=\conv K$, since
$\widehat K$ is then a convex set containing $K$. Moreover
$\partial\widehat K\subset K$, and hence $\partial\conv K\subset K$ in this
case. Therefore, if $T$ is hyponormal and $\widehat{\sigma(T)}$ is convex,
then $T$ has a nontrivial invariant subspace by Corollary \ref{c1}. This
polynomial-hull formulation covers, for instance, spectra obtained by removing
holes from a convex compact set.
\end{remark}

To provide explicit examples of hyponormal operators with nontrivial invariant subspaces by means of our criterion, we shall use the next realization result from \loccite{Corollary~1}{CP1}.
\begin{theorem}\label{p-putnam-realization}
Let $K\subset\CC$ be a nonempty compact set. Then $K$ is the spectrum of a completely non-normal hyponormal operator if and only if
\begin{equation}\label{locmeas}
    m_2(K\cap N)>0
\end{equation}
for every open disc $N\subset\CC$ such that $K\cap N\ne\emptyset$, where $m_2$ denotes planar Lebesgue measure.
\end{theorem}
\begin{remark}
\label{r-carey-pincus}
Theorem \ref{p-putnam-realization} has also a refined form.
In particular, the construction may be arranged with
rank-one self-commutator; see \loccite{Theorem~1 and Corollary~1}{CP1}.
This refinement is not needed below, except as evidence that the spectral
examples may be chosen close to normal operators. For the background behind these realization results, see \cite[Chs.~IX--XI]{MP}.
\end{remark}
Here ``completely non-normal'' means that the operator has no nonzero reducing subspace on which its restriction is normal.
Clearly, in view of the invariant subspace problem,
 the complete non-normality assumption can be made without loss of generality.

Thus Theorem \ref{p-putnam-realization} is an existence theorem.
It supplies many completely non-normal hyponormal operators satisfying Corollary \ref{c1}. Indeed, if a compact set $K$ satisfies \eqref{locmeas} 
and its polynomial hull $\widehat K$ is convex, then Theorem \ref{p-putnam-realization} gives a completely non-normal hyponormal operator $T$ with
$
    \sigma(T)=K,
$
and in view of Remark \ref{r-convex-spectrum}, 
$$
    \partial\conv\sigma(T)\subset\sigma(T).
$$
Hence $T$ has a nontrivial invariant subspace by Corollary \ref{c1}.

The next result is basic in the theory of invariant subspaces for hyponormal operators,
and we shall use it by comparing to Theorem \ref{t1} in several instances.
Here $C(K)$ denotes the algebra of continuous functions on a compact set
$K\subset\CC$, and $R(K)$ denotes the uniform closure on $K$ of rational
functions with poles off $K$.
\begin{theorem}[Brown]\label{thm:brown}
Let $T\in \BH$ be  hyponormal. If $R(\sigma(T))\ne C(\sigma(T))$, then $T$ has a nontrivial invariant subspace. In particular, the conclusion holds whenever $\operatorname{Int}\sigma(T)\ne\emptyset$.
\end{theorem}

We proceed with three examples illustrating how Theorem \ref{p-putnam-realization} and Corollary \ref{c1} apply in practice. The first two are also implied by Theorem \ref{thm:brown}, since their spectra have nonempty interior. The third is not covered by that theorem.
\begin{example}
The singular integral operators below, taken from \loccite{Section~1}{P70}, illustrate the realization picture much more concretely. See also \cite[Chapter II.2, Example 3]{MP}
for explanation of the role of these operators as models for general hyponormal operators.
We recall the form needed here. Let $q\in L^\infty(0,1), q \not \equiv 0,$ and let $M_x$ be multiplication by the independent variable $x$ on $L^2(0,1)$. Define the selfadjoint singular integral operator
$$
    (J_q f)(x)=\frac{1}{\pi i}\,q(x)\,\text{\rm p.v.}
    \int_0^1\frac{\overline{q(t)}f(t)}{t-x}\,dt,
    \qquad 0<x<1,
$$
where p.v. denotes the Cauchy principal value, and put
$$
    A_q=M_x+iJ_q.
$$
For these operators one has
$$
    A_q A_q^*-A_q^*A_q=2C_q,
    \qquad
    (C_q f)(x)=\frac{1}{\pi}q(x)
    \int_0^1\overline{q(t)}f(t)\,dt,
$$
so $C_q\ge0$. Thus $A_q$ is cohyponormal in our convention, $A_q^*$ is hyponormal,
and both have rank-one self-commutators.

In particular, let $\rho>0$ and take $q(x)\equiv\sqrt\rho$. Then
$$
    (A_\rho f)(x)=xf(x)+\frac{\rho}{\pi}\,\text{\rm p.v.}
    \int_0^1\frac{f(t)}{t-x}\,dt,
    \qquad 0<x<1,
$$
and in this case \loccite{Corollary to Theorem~4}{P70} gives
$$
    \sigma(A_\rho)=
    \{x+iy:0\le x\le1,\ -\rho\le y\le\rho\}.
$$
Consequently
$$
    T_\rho:=A_\rho^*
$$
is a (non-normal) hyponormal operator with
$$
    \sigma(T_\rho)=
    \{x+iy:0\le x\le1,\ -\rho\le y\le\rho\}.
$$
The spectrum is a rectangle, hence \(\partial\conv\sigma(T_\rho)\subset\sigma(T_\rho)\). The rectangle also satisfies the positivity condition in Theorem \ref{p-putnam-realization}, and Corollary \ref{c1} applies to $T_\rho$. Note that, moreover, $T_\rho$ 
is completely non-normal. 
Indeed, if a reducing subspace $M$ for $A_\rho$ carried a normal restriction, then $M$ would reduce $\operatorname{Re} A_\rho=M_x$. Since $M_x$ has simple spectrum, $M=L^2(E)$ for a measurable set $E\subset(0,1)$. Writing $\chi_E$ for the characteristic function of $E$, the compression of $C_\rho$ to $L^2(E)$ is
$$
    f\mapsto \frac{\rho}{\pi}\chi_E\int_E f(t)\,dt, \qquad f \in L^2(E),
$$
which vanishes only when $m(E)=0$. Thus the normal reducing part is zero. 
Since reducing subspaces for \(A_\rho\) and \(A_\rho^*\) coincide, the same
is true for \(T_\rho=A_\rho^*\).
By replacing $T_\rho$ with $\alpha+\beta T_\rho$, $\beta\neq0$, one obtains translated, rotated and dilated rectangles.
\end{example}

The next situation arises often in applications.

\begin{example}
\label{ex-annular-hyponormal}
Fix $0<r<1$ and put
\[
    K_r=\{z\in\mathbb C:r\le |z|\le1\}.
\]
Then $K_r$ has positive planar measure in every open disc meeting it and
\(
    \widehat K_r=\overline{\mathbb D}.
\)
Thus $\widehat K_r$ is convex and
\[
    \partial\conv K_r=\partial\widehat K_r=\mathbb T\subset K_r.
\]
Theorem \ref{p-putnam-realization} therefore gives completely non-normal
hyponormal operators $T$ with
\(
    \sigma(T)=K_r,
\)
and Corollary \ref{c1} applies to them. By the refinement recalled in Remark \ref{r-carey-pincus}, such annular spectra can also be realized with rank-one self-commutator.
\end{example}

The final example realizes a Cantor-like construction by hyponormal
operators. It starts with a fat Cantor set of radii and then removes a small
angular neighbourhood of the positive radius in order to make the inner
complement connected. The example shows that the present criterion is not
restricted to thick spectra and is not covered by Brown's theorem. We shall use
the following rational-approximation criterion from \cite[Ch.~VIII,
Corollary~8.4]{Gam}.
\begin{theorem}\label{thm:rational-criterion}
Let \(K\subset\CC\) be compact. Suppose that every point of \(\partial K\) belongs to the boundary of a component of \(\CC\setminus K\). Then \(R(K)=A(K)\).
In particular, $R(K)=C(K)$ whenever $\operatorname{Int}K=\emptyset$.
\end{theorem}

\begin{example}
\label{ex-non-dominating-hyponormal}
Choose numbers $0<\alpha_n<1$, $n\ge1$, such that
$\sum_{n\ge1}\alpha_n<\infty$. Starting with $E_0=[1/2,1]$, construct
$E_n$ inductively by deleting from the middle of each closed interval forming
$E_{n-1}$ an open interval of relative length $\alpha_n$, and put
\[
    E=\bigcap_{n\ge0}E_n .
\]
Then $E$ is compact and nowhere dense, $1\in E$, and
\[
    |E\cap I|>0
\]
whenever $I\subset\mathbb R$ is an open interval meeting $E$. Indeed, choose
$x_0\in I\cap E$. For each $N\ge0$, let $J_N$ be the closed interval forming
$E_N$ which contains $x_0$. Then
\[
    |E\cap J_N|=|J_N|\prod_{n>N}(1-\alpha_n)>0.
\]
Since $I$ is open and contains $x_0$, and since the lengths of the intervals
forming $E_N$ tend to zero as $N\to\infty$, we have $J_N\subset I$ for all
sufficiently large $N$. Hence, for such $N$,
\[
    |E\cap I|\ge |E\cap J_N|>0.
\]

Let $d_{\TT}(\xi,1)$ denote the angular distance on $\TT$ from $\xi$ to
$1$, that is,
\[
    d_{\TT}(\xi,1)=\inf\{|t-2\pi k|:\xi=e^{it},\ k\in\ZZ\} .
\]
Let $h:(0,1)\to(0,\pi/2)$ be a continuous function such that
$h(r)\to0$ as $r\uparrow1$.
 Put
\[
    G=\bigl\{z\in\mathbb D: |z|\notin E\bigr\}
\]
and
\[
    \mathcal N=\bigl\{r\xi:0<r<1,\ \xi\in\TT,\
    d_{\TT}(\xi,1)<h(r)\bigr\}.
\]
Set
\[
    \Omega=G\cup\mathcal N,
    \qquad K=\overline{\mathbb D}\setminus\Omega .
\]
The set $\Omega$ is open, since $E$ is closed and the functions
$z\mapsto |z|$, $z\mapsto d_{\TT}(z/|z|,1)$ on $\mathbb D\setminus\{0\}$,
and $h$ are continuous. It is dense in $\mathbb D$, because $E$ has empty
interior. It is also connected. Indeed, $\mathcal N$ is path connected: any
point of $\mathcal N$ can be joined to the positive radius by an arc of a
circle and then radially. The set $G$ is the union of the annuli corresponding
to the components of $(0,1)\setminus E$, together with the inner disc
corresponding to the component $(0,1/2)$. Each of these sets has nonempty
intersection with $\mathcal N$. Hence $\Omega$ is connected, by the elementary
fact that the union of a connected set with any family of connected sets each
having nonempty intersection with it is connected. Since $\Omega$ is dense in
$\mathbb D,$ it follows that $K$ has empty interior. Moreover,
\[
    \CC\setminus K=\Omega\cup(\CC\setminus\overline{\DD}),
\]
and these are the two complementary components.

We next prove that
\[
    R(K)=C(K).
\]
Since $\operatorname{Int}K=\emptyset$, we have $\partial K=K$. It remains to verify
the hypothesis of Theorem \ref{thm:rational-criterion}. If
$z=r\xi\in K\cap\mathbb D$, then $r\in E$. Since $E$ has empty interior,
there are radii $(r_j)_{j\ge1}\subset(0,1)\setminus E$ such that
$r_j\to r$ as $j\to\infty$. Hence $r_j\xi\in\Omega$ for all $j\ge1$, and
$r_j\xi\to z$. Thus $z\in\partial\Omega$. If $z\in\partial\mathbb D$, then
$z$ belongs to the boundary of the exterior component
$\mathbb C\setminus\overline{\mathbb D}$. Using Theorem
\ref{thm:rational-criterion}, we infer that $R(K)=A(K)$. Since
$\operatorname{Int}K=\emptyset$, this yields $R(K)=C(K)$.

We also have $\partial\mathbb D\subset K\subset\overline{\mathbb D}$, and
therefore
\[
    \conv K=\overline{\mathbb D},\qquad
    \partial\conv K=\partial\mathbb D\subset K.
\]
Let us show that $K$ has positive planar measure in every open disc meeting
it. Let $U$ be an open disc and
choose a point $\zeta\in U\cap K$.
Suppose first that $\zeta\in\mathbb D$. Write $\zeta=r_0\xi_0$, where
$\xi_0\in\TT$. Then $r_0\in E$ and $d_{\TT}(\xi_0,1)\ge h(r_0)$. There are
an interval $I\ni r_0$ and an arc $J\subset\TT$ containing $\xi_0$ such that
\[
    \{r\xi: r\in I,\ \xi\in J\}\subset U .
\]
Since $J$ is open and $d_{\TT}(\xi_0,1)\ge h(r_0)$, the continuity of
$d_{\TT}$ and, in the case of equality, a one-sided choice of the arc away
from $1$ allow us to choose a subarc $J_0\subset J$ of positive length such
that
\[
    \inf_{\xi\in J_0} d_{\TT}(\xi,1)>h(r_0).
\]
By the continuity of $h$, after shrinking $I$ we get
\[
    d_{\TT}(\xi,1)\ge h(r),\qquad r\in I,
    \quad \xi\in J_0 .
\]
Therefore
\[
    \{r\xi: r\in E\cap I,\ \xi\in J_0\}\subset K\cap U .
\]
Since $|E\cap I|>0$, $J_0$ has positive length, and $r\ge1/2$ on $E$, the polar-coordinate formula $dm_2=r\,dr\,d\theta$ shows that this set has positive planar measure.

It remains to treat the case where $\zeta\in\partial\mathbb D$. Since $U$ is
an open neighbourhood of $\zeta$, there are $\varepsilon>0$ and an arc
$J\subset\TT$ containing $\zeta$ such that
\[
    \{r\xi: 1-\varepsilon<r<1,
    \ \xi\in J\}\subset U .
\]
Choose a subarc $J_0\subset J$ of positive length such that
$1\notin\overline{J_0}$. By the continuity of $d_{\TT}$,
\[
    \gamma:=\inf_{\xi\in J_0}d_{\TT}(\xi,1)>0 .
\]
After decreasing $\varepsilon$ and using $h(r)\to0$ as $r\uparrow1$, we may
assume that
\[
    h(r)<\gamma,\qquad 1-\varepsilon<r<1 .
\]
Note that $(1-\varepsilon,1+\varepsilon)$ meets $E$. Since, as proved
above, every open interval meeting $E$ has positive intersection with $E$, and
since $E\subset[1/2,1]$, we infer
\[
    |E\cap(1-\varepsilon,1)|>0 .
\]
Thus
\[
    \{r\xi: r\in E\cap(1-\varepsilon,1),\ \xi\in J_0\}
\]
is contained in $K\cap U$ and has positive planar measure, again by the polar-coordinate formula 
and the fact that $r\ge1/2$ on $E$.

Then Theorem \ref{p-putnam-realization} gives a completely non-normal
hyponormal operator $T$ such that $\sigma(T)=K$. Since
$\partial\conv\sigma(T)\subset\sigma(T)$, Corollary \ref{c1} applies. On the
other hand, $R(\sigma(T))=C(\sigma(T))$, so the criterion in Theorem
\ref{thm:brown} does not apply.
Positive-area Swiss-cheese examples with $R(K)=C(K)$ are constructed in
\cite[p.~133]{BR} from the approximation principle \cite[Theorem]{BR}.
\end{example}

The disk example above was chosen for simplicity of presentation, since there $S(T)=\emptyset$. Examples with $S(T)\ne\emptyset$ are obtained
by the same construction with the closed disc replaced by a compact convex
polygon $P$ with $0\in\operatorname{int}P$. 

Since \(0\in\operatorname{int}P\), apart from the origin, every point of \(P\) has a unique representation
as \(r\xi\), where \(0< r\le1\) and \(\xi\in\partial P\). Repeating the
construction with this radial parametrization yields a 
compact set $K\subset P$ with
\[
    \partial P\subset K,
    \qquad \operatorname{Int}K=\emptyset,
    \qquad R(K)=C(K),
\]
which has positive planar measure in every open disc meeting it. The equality
$R(K)=C(K)$ follows from the same boundary-component check and Theorem~\ref{thm:rational-criterion}.

For the corresponding realized
hyponormal operator $T$ one has $\overline{W(T)}=P$ and hence
$S(T)=\partial P\subset\sigma(T)$.

The analogous assertion holds for cohyponormal operators, that is, for operators $T$ satisfying $T^*T\le TT^*$. Indeed, then $T^*$ is hyponormal. Moreover, $\sigma(T^*)=\{\overline z:z\in\sigma(T)\}$ and $W(T^*)=\{\overline z:z\in W(T)\}$, so the assumptions $S(T)\subset\sigma(T)$ and $S(T)=\emptyset$ are transferred to $T^*$ by complex conjugation. Applying Corollary \ref{c1} to $T^*$ gives a nontrivial invariant subspace $N$ for $T^*$. Hence $N^\perp$ is a nontrivial invariant subspace for $T$.

\section{Toeplitz operators}
Let $H^2:=H^2(\DD)$ be the Hardy space on the unit disc, identified with its boundary values in $L^2(\TT)$. Let $P:L^2(\TT)\to H^2$ be the orthogonal projection. For a symbol $\varphi\in L^\infty(\TT)$ the Toeplitz operator $T_\varphi$ on $H^2$ is defined by
\begin{equation}\label{defto}
T_\varphi f=P(\varphi f)\qquad(f\in H^2).
\end{equation}
The class of Toeplitz operators is basic in operator theory and one of the most well-studied. For an introduction to this subject, see, e.g., \cite{NikToeplitz}. 
Recall that
\begin{equation}\label{convexoid-toeplitz}
\overline{W(T_\varphi)}=\conv\sigma(T_\varphi),
\end{equation}
see \loccite{Corollary~3}{K}. Thus we have the following corollary.

\begin{corollary}\label{c2}
Let $T_\varphi$ be a Toeplitz operator with $\varphi\in L^\infty(\TT)$. If
$$
S(T_\varphi)\subset\sigma(T_\varphi),
$$
then $T_\varphi$ has a nontrivial invariant subspace. In particular, this holds whenever $S(T_\varphi)=\emptyset$.
Equivalently, the same conclusion holds under the condition
$$
\partial\conv\sigma(T_\varphi)\subset\sigma(T_\varphi).
$$
\end{corollary}

\begin{proof}
The proof is the same as that of Corollary \ref{c1}, using \eqref{convexoid-toeplitz} in place of \eqref{convexoid-hyp}.
\end{proof}

If $\varphi$ is a continuous function on $\TT$, then the spectrum of $T_\varphi$ is given by the standard Toeplitz spectrum formula \loccite{Corollary~3.1.7}{NikToeplitz}:
\begin{equation} \label{spformula}
\sigma(T_\varphi)=\mathcal R(\varphi)\cup\{\lambda\in\CC\setminus\mathcal R(\varphi): w(\varphi,\lambda)\ne 0\},
\end{equation}
where $\mathcal R(\varphi)$ denotes the range of $\varphi$ and $w(\varphi,\lambda)$ is the winding number of the curve determined by $\varphi$ with respect to $\lambda$.

Next, we apply our results to the study of invariant subspaces of Toeplitz operators.
Let us recall the type of regularity entering Peller's invariant-subspace criteria for Toeplitz operators. The continuous-symbol results below are due to Peller: the Lipschitz-arc version is from \cite{PellerToeplitz83}, and the $C^2$-arc refinement from \cite{PellerToeplitz87}; see also \cite{PellerToeplitz} for the discussion and for piecewise continuous symbols. To this end, for $\varphi \in C(\mathbb T)$ denote by  $\omega_\varphi(\delta):=$
$$
 \sup\{|\varphi(e^{is})-\varphi(e^{it})|:|s-t|\le\delta\}, \qquad \delta \ge 0,
$$
 its modulus of continuity.
\begin{theorem}
Let $T_\varphi \in \mathcal B(H^2(\mathbb D))$ be given by \eqref{defto} with 
$\varphi \in C(\mathbb T)$, $\varphi\ne {\rm const}$. Suppose that there are a Jordan arc $J$ and an open disc $D\subset\CC$ such that
$$
        \varphi(\mathbb T)\cap J\cap D\ne\emptyset,
        \qquad
        \varphi(\mathbb T)\cap(D\setminus J)=\emptyset.
$$
If $J$ is Lipschitz and, for some $\delta_0>0$,
$$
        \int_0^{\delta_0}
        \frac{\omega_\varphi(\delta)}
             {\delta\log(1/\delta)}\,d\delta<\infty,
$$
then $T_\varphi$ has a nontrivial invariant subspace. If $J$ is of class $C^2$, the same conclusion holds under the weaker square Dini-log condition
\begin{equation}\label{dini_log}
        \int_0^{\delta_0}
        \frac{(\omega_\varphi(\delta))^2}
             {\delta\log(1/\delta)}\,d\delta<\infty.
\end{equation}
\end{theorem}
These results are, to the best of our knowledge, the strongest general existence criteria of this type in terms of symbol regularity. The criterion below is of a different nature: it uses only the spectral geometry, which for continuous symbols is determined by the range and the winding number in the formula above.

Thus there are many functions $\varphi$ such that the Toeplitz operator $T_\varphi$ satisfies the condition of Corollary \ref{c2}. The following class is particularly transparent.

\begin{example}\label{ex-convex-jordan-toeplitz}
Let $\Gamma$ be a Jordan curve which bounds a convex domain $\Omega$, and let $\varphi:\TT\to\Gamma$ be a continuous parametrization with nonzero winding number around the points of $\Omega$ and winding number zero on $\CC\setminus\overline\Omega$. Then the preceding spectrum formula \eqref{spformula} gives
\begin{equation}\label{spectrum}
    \sigma(T_\varphi)=\overline\Omega.
\end{equation}
Indeed, the range term gives $\Gamma$, and the winding-number term fills precisely the interior $\Omega$. Since $\overline\Omega$ is convex, we have
$$
\partial\conv\sigma(T_\varphi)=\partial\sigma(T_\varphi)=\Gamma\subset\sigma(T_\varphi).
$$
Thus Corollary \ref{c2} applies. In this way every convex Jordan curve, with any continuous nonzero-index parametrization, gives a Toeplitz operator whose spectrum contains $\partial\operatorname{conv}\sigma(T_\varphi)$, and hence has a nontrivial invariant subspace by Corollary \ref{c2}.

A concrete subfamily is obtained by taking ellipses parametrised by
$$
    \varphi_a(e^{it})=e^{it}+a e^{-it}=(1+a)\cos t+i(1-a)\sin t,\qquad t \in [0,2\pi)
$$
for $0<a<1.$ Then
$$
    T_{\varphi_a}=T_z+aT_{\overline z}=S+aS^*,
$$
where $S=T_z$ is the unilateral shift on $H^2$, 
and, in view of \eqref{spectrum},
$$
\partial \sigma(T_{\varphi_a})= \varphi_a(\mathbb T).
$$

This recovers the explicit example $S+aS^*$ as a special case of the convex-curve construction.

The same ellipse gives rough continuous symbols to which Corollary \ref{c2} applies, although the Dini-log conditions above fail.
Choose $t_0>0$ and $c>0$ so small that
$$
    c\bigl(\log\log(e^e/t_0)\bigr)^{-1/2}<\pi/4.
$$
Let $u:[0,2\pi]\to[0,2\pi]$ be an increasing homeomorphism such that
$u(0)=0$, $u(2\pi)=2\pi$, and
$$
    u(t)=c\bigl(\log\log(e^e/t)\bigr)^{-1/2},
    \qquad 0<t\le t_0 .
$$
Define \(\psi:\mathbb T\to \varphi_a(\mathbb T)\) by
$$
    \psi(e^{it}):=\varphi_a(e^{iu(t)}), \qquad t \in [0,2\pi].
$$

Then \(\psi\) is continuous, has range \(\varphi_a(\mathbb T)\), and winds once around the interior of the ellipse \(\varphi_a(\mathbb T)\).
Hence by \eqref{spectrum},
\[
        \partial\sigma(T_\psi)=\partial\sigma(T_{\varphi_a})
        =\varphi_a(\mathbb T),
\]
and Corollary \ref{c2} again applies to \(T_\psi\).

It remains to check that this rough parametrization is not covered by the relevant
modulus-of-continuity hypotheses.
Indeed, the parametrized ellipse satisfies
$$
    \frac d{ds}\varphi_a(e^{is})\bigg|_{s=0}=i(1-a)\ne0,
$$
so, for some $c_1>0$ and all sufficiently small $s>0$,
$$
    |\varphi_a(e^{is})-\varphi_a(1)|\ge c_1s.
$$
Therefore, for all sufficiently small $\delta>0$,
$$
    \omega_\psi(\delta)
    \ge |\psi(e^{i\delta})-\psi(1)|
    = |\varphi_a(e^{iu(\delta)})-\varphi_a(1)|
    \ge C\bigl(\log\log(e^e/\delta)\bigr)^{-1/2},
$$
and then
$$
 \int_0^{\delta_0}
 \frac{(\omega_\psi(\delta))^2}
   {\delta\log(1/\delta)}\,d\delta=\infty
$$
for all $\delta_0\in (0,1).$
Thus the square Dini-log condition \eqref{dini_log} fails. Since $\omega_\psi(\delta)\to0$ as $\delta\downarrow0$, this also implies the failure of the stronger Dini-log condition without the square. Consequently this example is not covered by either of the Peller criteria recalled above.
\end{example}

\section{A remark on bilateral weighted shifts}
\label{sec:annular-bilateral}

Let $B_w$ be a bilateral weighted shift on $\ell^2(\mathbb Z)$ with respect
to the canonical basis $(e_n)_{n\in\mathbb Z}$,
\[
    B_we_n=w_ne_{n+1},\qquad n\in\mathbb Z,
\]
with the weight $(w_n)_{n \in \mathbb Z}\subset \mathbb C.$
We assume throughout this section that $B_w$ is invertible, equivalently
\[
    0<\inf_{n\in\mathbb Z}|w_n|\le \sup_{n\in\mathbb Z}|w_n|<\infty .
\]
After a diagonal unitary conjugation, see \cite[Corollary~1]{Sh}, we may
suppose that $w_n>0$ for all $n\in\mathbb Z$ whenever formulas involving the
weights are used.

The coordinate tails already give nontrivial invariant subspaces for $B_w$.
For invertible bilateral shifts the more relevant question is two-sided
invariance, or hyperinvariance. Recall that a subspace is hyperinvariant for
an operator $T$ if it is invariant under every bounded operator commuting with
$T$; for an invertible $B_w$, such a subspace is automatically invariant under
all powers $B_w^n$, $n\in\mathbb Z$.

The argument below is independent of Theorem \ref{t1}: the relevant spectral
sets are annuli rather than convex sets. We shall use the following annular
spectral-set theorem of Crouzeix and Greenbaum, see
\loccite{Theorems~10 and~11}{CG}. For an operator $T\in\BH$ let
\[
    \omega(T)=\sup\{|\langle Tx,x\rangle|:\|x\|=1\}
\]
denote its numerical radius.

\begin{theorem}\label{thm:annular-CG}
Let
\[
    \mathcal A_s=\{z\in\mathbb C:s^{-1}\le |z|\le s\},\qquad s>1.
\]
If $T\in \BH$ is invertible and satisfies
\[
    \omega(T)\le s,
    \qquad
    \omega(T^{-1})\le s,
\]
then there exists a constant $K(s)$ such that $\mathcal A_s$ is a
$K(s)$-spectral set for $T$.
\end{theorem}

In addition, we shall use the spectral-set result due to 
Chevreau--Pearcy--Shields \cite[Theorem 6.2]{CPS}, formulated below.
\begin{theorem}\label{thm:CPS-shift}
Let $T$ be an invertible bilateral weighted shift on $\ell_2(\mathbb Z).$ If $\sigma(T)$ is a
spectral set for $T$, then $T$ has a nontrivial hyperinvariant subspace.
\end{theorem}
The next simple proposition is essentially a combination of the two theorems above.
\begin{proposition}\label{prop:annular-hyperinvariant-shift}
Let $B_w$ be an invertible bilateral weighted shift on $\ell^2(\mathbb Z)$.
Suppose that
\[
    \omega(B_w)=r(B_w),\qquad \omega(B_w^{-1})=r(B_w^{-1}).
\]
Then $B_w$ has a nontrivial hyperinvariant subspace.
\end{proposition}

\begin{proof}
Put
\[
    R=r(B_w),\qquad \rho=\frac1{r(B_w^{-1})},\qquad
    c=(R\rho)^{1/2},\qquad s=\left(\frac R\rho\right)^{1/2},
\]
and set $T_w=c^{-1}B_w$. By the standard spectral description of bilateral
weighted shifts, see \cite[Theorem~5a]{Sh},
\[
    \sigma(B_w)=\{z\in\mathbb C:\rho\le |z|\le R\}.
\]
Therefore
\[
    \sigma(T_w)=\{z\in\mathbb C:s^{-1}\le |z|\le s\}.
\]
Assume first that $R>\rho$, so that $s>1$. Then
\[
    \omega(T_w)=s,
    \qquad
    \omega(T_w^{-1})=s.
\]
By Theorem \ref{thm:annular-CG}, $\sigma(T_w)=\mathcal A_s$ is a spectral set
for $T_w$.

It remains to consider the circular case $R=\rho$. Then $s=1$ and
$\sigma(T_w)=\mathbb T$. In this case
\[
    W(T_w)\subset \overline{\mathbb D},
    \qquad
    W(T_w^{-1})\subset \overline{\mathbb D},
\]
and therefore $T_w$ is unitary by \cite[Corollary~1]{StampfliThin}. Thus, in
the circular case as well, $T_w$ is normal, and so $\sigma(T_w)$ is a spectral
set for $T_w$.

Scaling back, $\sigma(B_w)$ is a spectral set for $B_w$. Theorem
\ref{thm:CPS-shift} now gives a nontrivial hyperinvariant subspace for $B_w$.
\end{proof}

The corresponding norm-annulus statement is already contained in
\cite[Theorem~6.2]{CPS}.

\section{Final comments}\label{sec:final-comments}
Let us indicate the relation with \cite[Remark]{KL}. That remark shows
that if \(\widehat{\sigma(T)}\) is a spectral set for \(T\), and if
\(\partial\widehat{\sigma(T)}\) is a Jordan curve, then \(T\) has a
nontrivial invariant subspace.
 By Lemma
\ref{lem:hull-numerical-range}, in the non-degenerate case the hypothesis of
Theorem \ref{t1} is equivalent to
\[
        \widehat{\sigma(T)}=\overline{W(T)}.
\]
Since $\|p\|_{\sigma(T)}=\|p\|_{\widehat{\sigma(T)}}$ for polynomials,
Theorem \ref{crouzeix} supplies the polynomial estimate required in
\cite[Remark]{KL}. Moreover, the outer boundary of $\sigma(T)$ is then
$\partial\widehat{\sigma(T)}=\partial W(T)$, the boundary of a compact
convex set with nonempty interior, and hence a Jordan curve. Thus, apart from
the degenerate case already treated in the proof, Theorem \ref{t1} also
follows from \cite[Remark]{KL} (with the conformal step extended as in the proof of Theorem \ref{t1}). The examples in this paper emphasize further applications and a complementary
realization point: in natural classes such as hyponormal and Toeplitz operators
one can produce operators satisfying this geometry explicitly, including
examples not covered by Brown's theorem or by Peller's Toeplitz criteria.
The statement on bilateral shifts is of the same spirit, but it is direct and 
does not rely on conformal mappings machinery.

Note also that the appearances of $\conv\sigma(T)$ above rely on the
convexoid identities $\overline{W(T)}=\conv\sigma(T)$ available in the
hyponormal and Toeplitz settings. For a general operator the numerical range
cannot be replaced by the convex hull of the spectrum: for a nonzero square-zero
operator $N$ one has $\widehat{\sigma(N)}=\conv\sigma(N)=\{0\}$, whereas
$W(N)$ has nonempty interior.

\end{document}